\documentclass{amsart}
\usepackage[margin=1.5in]{geometry}

\usepackage{amsmath,amsthm, latexsym, amssymb,mathrsfs,  tikz-cd}
\usepackage{stackengine,scalerel, quiver, MnSymbol}
\usepackage[colorlinks]{hyperref}
\usepackage[capitalize]{cleveref}
\definecolor{dark-red}{rgb}{0.6,0,0}
\definecolor{dark-green}{rgb}{0,0.4,0}
\definecolor{medium-blue}{rgb}{0,0,0.5}
\hypersetup{
    colorlinks, linkcolor={dark-red},
    citecolor={dark-green}, urlcolor={medium-blue}
}

\newcommand{\an}{\mr{an}}

\newcommand{\gr}{\mr{gr}}
\newcommand{\Gr}{\mathrm{Gr}}

\newcommand{\Fl}{\mr{Fl}}

\newcommand{\Spd}{\mathrm{Spd}}
\newcommand{\GL}{\mathrm{GL}}
\newcommand{\HT}{\mr{HT}}

\newcommand{\mc}[1]{\mathcal{#1}}
\newcommand{\mbb}[1]{\mathbb{#1}}
\newcommand{\mr}[1]{\mathrm{#1}}
\newcommand{\mf}[1]{\mathfrak{#1}}

\newcommand{\dR}{\mathrm{dR}}
\newcommand{\BC}{\mathrm{BC}}

\newcommand{\adm}{\mathrm{adm}}

\newcommand{\ab}{\mathrm{ab}}\newcommand{\der}{\mathrm{der}}

\newcommand{\qp}{\mathbb{Q}_p}
\newcommand{\qpbar}{\overline{\mathbb{Q}}_p}
\newcommand{\qpbreve}{\breve{\mathbb{Q}}_p}

\DeclareMathOperator{\Lie}{Lie}

\newcommand{\ul}[1]{\underline{#1}}

\newcommand{\Hdg}{\mr{Hdg}}

\newcommand{\vsp}{\mathrm{vsp}}
\newcommand{\nvsp}{\mathrm{non-vsp}}

\newcommand{\sm}{\mathrm{sm}}
\newcommand{\End}{\mr{End}}
\numberwithin{equation}{subsection}
\numberwithin{equation}{subsubsection}

\theoremstyle{plain}

\newtheorem*{theorem*}{Theorem}

\newtheorem{theorem}[subsubsection]{Theorem}
\newtheorem{corollary}[subsubsection]{Corollary}
\newtheorem{conjecture}[subsubsection]{Conjecture}
\newtheorem{proposition}[subsubsection]{Proposition}
\newtheorem{lemma}[subsubsection]{Lemma}

\theoremstyle{definition}

\newtheorem{remark}[subsubsection]{Remark}

\newcommand{\fpbar}{\overline{\mathbb{F}}_p}
\newcommand{\alg}{\mr{alg}}

\newcommand{\M}{\mathscr{M}}

\newcommand{\lf}{\mathrm{lf}}
\newcommand{\lfid}{{\diamond_{\lf}}}

\title[A cohomological smoothness conjecture]{A cohomological smoothness conjecture for moduli of mixed characteristic local shtukas with one leg}

\author{Sean Howe}
\begin{document}

\begin{abstract}  We give a simple geometric characterization of the locus where the inscribed Banach--Colmez Tangent Spaces of moduli of mixed characteristic local shtukas with one leg and fixed determinant are connected. We conjecture that the structure morphism for the underlying diamond is cohomologically smooth over this locus and, applying the Fargues-Scholze Jacobian criterion, we prove this conjecture in the case of EL infinite level Rapoport-Zink spaces, generalizing a result of Ivanov--Weinstein in the basic case.
\end{abstract}

\maketitle

\tableofcontents

\section{Introduction}
Let $E/\mathbb{Q}_p$ be a finite extension, let $G/E$ be a connected reductive group, let $[\mu]$ be a conjugacy class of cocharacters of $G_{\overline{E}}$ with field of definition $E([\mu])$, and let $b \in G(\breve{E})$ represent a class in the Kottwitz set $B(G,[\mu^{-1}])$. Attached to this data there is, as in \cite{ScholzeWeinstein.BerkeleyLecturesOnPAdicGeometryAMS207, HoweKlevdal.AdmissiblePairsAndpAdicHodgeStructuresIITheBiAnalyticAxLindemannTheorem, Howe.InscriptionTwistorsAndPAdicPeriods}, a diamond infinite level moduli of mixed characteristic local shtukas with one leg 
\[ \M_{b,[\mu]} / \Spd \breve{E}([\mu]) \]
equipped with lattice Hodge and Hodge-Tate period maps to Schubert cells in the $\mathbb{B}^+_\dR$-affine Grassmannian of $G$,
\[ \pi_{\Hdg}^+: \M_{b,[\mu]} \rightarrow \Gr_{[\mu], \breve{E}([\mu])} \textrm{ and } \pi_{\HT}^+: \M_{b,[\mu]} \rightarrow \Gr_{[\mu^{-1}], \breve{E}([\mu])}. \]
It is also equipped with commuting actions of $G(E^\diamond)$ (where $E^\diamond$ is the topological constant sheaf usually denoted $\ul{E}$) and a more exotic group diamond $\tilde{G}_b$. The period map $\pi_{\HT}^+$ is a $G(E^\diamond)$-equivariant $\tilde{G}_b$-torsor over its image, the Newton stratum $\Gr_{[\mu^{-1}]}^{[b]}$ (a locally closed subdiamond), and the period map $\pi_{\Hdg}^+$ is a $\tilde{G}_b$-equivariant $G(\mathbb{Q}_p^\diamond)$-equivariant torsor over its image, the $b$-admissible locus $\Gr_{[\mu]}^{b-\adm}$ (an open subdiamond).

 When $[\mu]$ is minuscule, $\M_{b,[\mu]}$ is a diamond infinite level local Shimura variety, $\Gr_{[\mu]}^{b-\adm}$ is the diamond associated to an open $\Fl_{[\mu^{-1}]}^{b-\adm}$ in the rigid analytic flag variety $\Fl_{[\mu^{-1}]}$, and $\M_{b,[\mu]}$ is an inverse limit of the diamonds associated to \'{e}tale covers $\M_{b,[\mu],K_p}^{\mr{rig}}$ of $\Fl_{[\mu]}^{b-\adm}$ indexed by compact open subgroups $K_p \leq G(\mathbb{Q}_p)$ such that $\M_{b,[\mu]}/K_p = (\M_{b,[\mu]}^{\mr{rig}})^\diamond.$
These local Shimura varieties, whose existence was first predicted in \cite{RapoportViehmann.TowardsATheoryOfLocalShimuraVarieties}, include the rigid generic fibers of Rapoport--Zink \cite{RapoportZink.PeriodSpacesForpDivisibleGroups} moduli spaces of $p$-divisible groups with extra structures.

Let $\mathbb{C}_p=\overline{E}^\wedge$. In \cite{IvanovWeinstein.TheSmoothLocusInInfiniteLevelRapoportZinkSpaces}, Ivanov and Weinstein used the Fargues--Scholze \cite{FarguesScholze.GeometrizationOfTheLocalLanglandsCorrespondence} Jacobian criterion  to construct a large cohomologically smooth locus in a component of $\M_{b,[\mu],\mbb{C}_p}$ when this space is the infinite level Rapoport--Zink space associated to an EL moduli problem and $b$ is basic (in particular, they work with $E=\mbb{Q}_p$). Cohomological smoothness is a natural finiteness property for $\ell$-adic cohomology and this result provided, in particular, a conceptual explanation for earlier cohomological finiteness results for infinite level Lubin-Tate space obtained by Weinstein \cite{Weinstein.SemistableModelsForModularCurvesOfArbitraryLevel} through the construction of explicit formal models.

In this work, we extend the result of Ivanov--Weinstein to the full infinite level EL Rapoport--Zink case (i.e. we remove the basic hypothesis), and conjecture a generalization to all $\M_{b,[\mu]}$. The proof combines the method of \cite{IvanovWeinstein.TheSmoothLocusInInfiniteLevelRapoportZinkSpaces} with the computation of Banach--Colmez Tangent Bundles via inscription in \cite{Howe.InscriptionTwistorsAndPAdicPeriods}, and this computation leads to our general conjecture. 

\subsection{A conjecture and a theorem}

\subsubsection{}The formation of the spaces $\M_{b,[\mu]}$ is functorial in the data, so, for $\det:G \rightarrow G^{\mr{ab}}$ the canonical map from $G$ to its abelianization, we obtain 
\[ \det_{b,[\mu]}: \M_{b,[\mu]} \rightarrow \M_{\det(b), [\det\circ \mu], \breve{E}([\mu])}.\]
After base-change to $\mathbb{C}_p$, the space on the right is a locally profinite set (it is a torsor for $G^{\mr{ab}}(E)$ over $\Spd \breve{E}([\mu])$), so we may fix a point $\tau: \Spd \mathbb{C}_p \rightarrow \M_{\det(b), [\det\circ \mu], \breve{E}([\mu])}$. We write $\M_{b,[\mu]}^\tau$ for the fiber over $\tau$, a diamond over $\Spd \mathbb{C}_p$, and consider the induced period maps
\[ \pi_{\Hdg}^{+,\tau}: \M_{b,[\mu]}^\tau \rightarrow \Gr_{[\mu], \mathbb{C}_p} \textrm{ and } \pi_{\HT}^{+,\tau}: \M_{b,[\mu]}^\tau \rightarrow \Gr_{[\mu^{-1}], \mathbb{C}_p} . \]

\subsubsection{} In \cref{cor.vsp-closed}, we show there is a closed sub-diamond $\M^{\tau,\vsp}_{b,[\mu]}$, the very special locus, characterized by the following property: a rank one geometric point $z: \Spd(C) \rightarrow \M_{b,[\mu]}^\tau$ factors through $\M^{\tau,\vsp}_{b,[\mu]}$ if and only if $z$ not an isolated point in the intersection of its fibers for the two period maps, $(\pi_\HT^{+,\tau})^{-1}(\pi_{\HT}^{+,\tau}(z)) \cap (\pi_{\Hdg}^{+,\tau})^{-1}(\pi_{\Hdg}^{+,\tau}(z)).$

To make sense of this condition concretely, note that, since $\pi_{\Hdg}^{+,\tau}$ is a $G^\der(E)$-torsor (\cref{lemma.tau-torsor}), the orbit map of $z$ identifies $(\pi_{\Hdg}^{+,\tau})^{-1}(\pi_{\Hdg}^{+,\tau}(z))$ with the locally profinite set $G^\der(E)$, and then the intersection with the fiber of $\pi_{\HT}^+$ is a closed subset of $G^\der(E)$. It is in this topology that we ask for $z$ not to be isolated; note that this topology is the same as the diamond topology on the diamond-theoretic fiber of $\pi_{\HT}^{+,\tau} \times \pi_\Hdg^{+,\tau}$ above $\pi_{\HT}^{+,\tau}(z) \times \pi_\Hdg^{+,\tau}(z)$. 

\begin{remark}
    We note that $z$ being isolated is equivalent to the intersection of the fibers being discrete (this follows, e.g., from \cref{lemma.fiber-orbit}), so the condition only depends on the two periods and not on the choice of $z$ in the intersection of the fibers. 
\end{remark}

 We write $\M_{b,[\mu]}^{\tau,\nvsp}$ for the open complement of $\M_{b,[\mu]}^{\tau,\vsp}$. 
\begin{conjecture}\label{conj.cohom-smooth}
    The structure morphism $\M_{b,[\mu]}^{\tau,\nvsp} \rightarrow \Spd \mbb{C}_p$ is cohomologically smooth. 
\end{conjecture}

\begin{theorem}\label{theorem.minuscule-EL}(see \cref{corollary.non-sp-locus})
\cref{conj.cohom-smooth} holds in the EL Rapoport-Zink case. 
\end{theorem}

Here the EL Rapoport-Zink case is as in \cite[\S2.2]{IvanovWeinstein.TheSmoothLocusInInfiniteLevelRapoportZinkSpaces}, recalled at the beginning of \cref{s.EL-case}; in particular, we work over $E=\mathbb{Q}_p$ in \cref{theorem.minuscule-EL} (but this is not a serious constraint).

\begin{remark}
The case of \cref{theorem.minuscule-EL} when $b$ is basic is equivalent to \cite[Theorem 1.0.1]{IvanovWeinstein.TheSmoothLocusInInfiniteLevelRapoportZinkSpaces}, though this is not immediately obvious because the definition of the very special\footnote{In \cite{IvanovWeinstein.TheSmoothLocusInInfiniteLevelRapoportZinkSpaces}, this is just called the special locus. We use the term very special instead of special because, from a Tannakian perspective as in \cite{HoweKlevdal.AdmissiblePairsAndpAdicHodgeStructuresIIIVariationAndUnlikelyIntersection}, it is natural to reserve the word special for the locus of points where there is any reduction of the Tannakian structure group from the generic structure group --- in general, such a reduction can come from the existence of other Hodge tensors besides endomorphisms. Our terminology is not perfect: outside the basic case, very special does not actually imply special in this Tannakian sense! For example, the endomorphisms of an ordinary $p$-divisible group over $\mathcal{O}_{\mathbb{C}_p}$ may not correspond to endomorphisms of the associate Tannakian object parameterized by a point in the Serre-Tate moduli space, which records an added rigidification (cf., e.g., \cite[\S6.3]{HoweKlevdal.AdmissiblePairsAndpAdicHodgeStructuresITranscendenceOfTheDeRhamLattice} for related considerations).} locus given there is different: it is defined in loc. cit. as the locus consisting of the geometric points $z$ such that there are extra endomorphisms of the $p$-divisible group parameterized by $z$. We show these conditions are equivalent in this case (\cref{lemma.vsp-pdiv}); in general, our geometric very special condition is equivalent to the associated modification of $G$-bundles admitting extra infinitesimal automorphisms (cf. \cref{prop.HN-slopes-condition}).
\end{remark}

\subsection{Tangent Spaces and proof strategy}
The proof of \cite[Theorem 1.0.1]{IvanovWeinstein.TheSmoothLocusInInfiniteLevelRapoportZinkSpaces} proceeds as follows: first, they construct a smooth quasi-projective variety $Z$ over the Fargues-Fontaine curve $X_{\mathbb{C}_p^\flat}$ and show that $\M_{b,[\mu]}^\tau$ is isomorphic to the moduli of sections of $Z$. Using their explicit construction of $Z$, they then give a presentation of the pull-back of the relative tangent bundle $s^*T_{Z/X_{\mathbb{C}_p^\flat}}$ along any section $s$. Finally, they analyze this presentation to show that it has positive slopes exactly when the corresponding point parameterizes a $p$-divisible group that does not have extra endomorphisms (this is accessible via the Hodge-Tate filtration using the Scholze-Weinstein classification \cite{ScholzeWeinstein.ModuliOfpDivisibleGroups}). The Fargues-Scholze Jacobian criterion then implies that this locus is open and cohomologically smooth. 

The above strategy applied in the infinite level basic EL Rapoport-Zink case. In \cite{Howe.InscriptionTwistorsAndPAdicPeriods}, we upgraded \emph{all} of the moduli spaces $\M_{b,[\mu]}$ to \emph{inscribed $v$-sheaves}, which is a differential enrichment of the theory of diamonds, then computed their internal tangent bundles within that theory by quite different means. One easily deduces an inscribed structure on $\M_{b,[\mu]}^\tau$, and a computation of its inscribed tangent bundle. Restricting back to the underlying $v$-sheaf gives rise to a Banach-Colmez Tangent Bundle that is the space of global sections of a vector bundle $\mc{E}^\circ_{\mr{max}}$ on the relative Fargues-Fontaine curve with non-negative slopes. Using this, we are able to break up the general problem into two different steps.

First, we show $\M_{b,[\mu]}^{\tau,\vsp}$ is precisely the locus consisting of geometric points where $\mc{E}^\circ_{\mr{max}}$ admits zero as a Harder-Narasimhan slope (i.e. the set of points that have a locally profinite tangent direction) --- in particular, since $\mc{E}^\circ_{\mr{max}}$ has non-negative Harder-Narasimhan slopes, this is a closed subdiamond. This analysis is carried out in \cref{s.analysis}, culminating in \cref{prop.HN-slopes-condition} and \cref{cor.vsp-closed}. The starting point is the computation of the tangent bundles in \cite[\S9.4]{Howe.InscriptionTwistorsAndPAdicPeriods}, but the comparison with the geometric condition also requires a duality trick to obtain information about the tangent bundle of the inscribed intersection (it is here that we use the reductive hypothesis) and then a $p$-adic Lie group trick to pass information about the inscribed tangent bundle of the intersection back into its geometry.

If we knew that the Jacobian criterion applied to these Tangent Bundles, then this would essentially\footnote{Because we obtain an identity of Banach-Colmez spaces of global sections rather than vector bundles, we need a small additional argument to apply the Jacobian criterion; see the proof of \cref{corollary.non-sp-locus}.} prove \cref{conj.cohom-smooth}. Thus, in order to obtain \cref{theorem.minuscule-EL}, it suffices to show the Jacobian criterion applies in the  Rapoport-Zink EL case. But in \cite{Howe.InscriptionTwistorsAndPAdicPeriods} we have also shown that, for a quasi-projective scheme $Z/X_{\mathbb{C}_p^\flat}$ as above, one obtains an inscribed moduli of sections, and the associated Tangent Bundle is precisely the one appearing in the Jacobian criterion. Thus we can conclude if we can show the inscribed moduli of sections construction associated to some smooth quasi-projective $Z$ produces the same inscribed structure obtained from \cite{Howe.InscriptionTwistorsAndPAdicPeriods}. But it is almost formal that the construction of $Z$ in \cite{IvanovWeinstein.TheSmoothLocusInInfiniteLevelRapoportZinkSpaces} accomplishes this in the general Rapoport-Zink EL setting! We explain this in \cref{s.EL-case}. 

\begin{remark} Note that in this proof we never actually need to compare our presentation of the Tangent Bundle to the presentation given in \cite{IvanovWeinstein.TheSmoothLocusInInfiniteLevelRapoportZinkSpaces} --- that the Tangent Bundles are the same is given to us automatically by comparing the inscribed structures. In fact, the reason we can extend the result of \cite{IvanovWeinstein.TheSmoothLocusInInfiniteLevelRapoportZinkSpaces} beyond the basic case is essentially because the presentation of the Tangent Bundle coming from \cite{Howe.InscriptionTwistorsAndPAdicPeriods} is simpler to work with than the presentation in \cite{IvanovWeinstein.TheSmoothLocusInInfiniteLevelRapoportZinkSpaces}! 
\end{remark}

\begin{remark}
    When $G^\mr{ab} \neq 0$, one must pass to $\M_{b,[\mu]}^\tau$ to obtain an interesting result because the action of the center $Z(G)(E)$ on $\M_{b,[\mu]}$ contributes locally profinite tangent directions at every point --- this is reflected in the failure of cohomological smoothness for any non-empty open subdiamond, which occurs because, using the same action and $\det_{b,[\mu]}$, one finds that any such subdiamond has infinitely many geometric connected components.  
\end{remark}

\subsection{Acknowledgements}We thank Christian Klevdal, Gilbert Moss, and Peter Wear for participating in a reading group on \cite{IvanovWeinstein.TheSmoothLocusInInfiniteLevelRapoportZinkSpaces} that ultimately led to this work. We thank Alexander Ivanov and Jared Weinstein for helpful conversations about the results of \cite{IvanovWeinstein.TheSmoothLocusInInfiniteLevelRapoportZinkSpaces}.

During the preparation of this work, the author was supported by the National Science Foundation through grants DMS-2201112 and DMS-2501816. The author also received support as a visitor at the 2023 Hausdorff Trimester on The Arithmetic of the Langlands Program at the Hausdorff Research Institute for Mathematics from the Deutsche Forschungsgemeinschaft (DFG, German Research Foundation) under Germany’s Excellence Strategy– EXC-2047/1–390685813, as a member at the Institute for Advanced Study during the academic year 2023-24 from the Friends of the Institute for Advanced Study Membership, and during the academic year 2024-2025 from the University of Utah Faculty Fellow program.

\section{Analysis of the tangent bundle}\label{s.analysis}
In this section we will analyze the Tangent Bundle of $\M_{b,[\mu]}^\tau$; in particular, we will show the pointwise condition defining the non-very-special locus is equivalent to the Tangent Space at that point being connected (see \cref{prop.HN-slopes-condition}/\cref{remark.connected}). 

\subsection{Setup}
Let $E/\mbb{Q}_p$ be a finite extension with residue field $\mbb{F}_q$, fix an algebraic closure $\overline{E}$ with completion $\mathbb{C}_p:=\overline{E}^\wedge$, and let $\breve{E}$ be the completion in $
\mathbb{C}_p$ of the maximal unramified extension of $E$ in $\overline{E}$.

Let $G/E$ be a connected reductive group, let $b \in G(\breve{E})$, and let $[\mu]$ be a conjugacy class of cocharacters of $G_{\overline{E}}$ such that $b$ represents a class in the Kottwitz set $B(G,[\mu^{-1}])$.  We write $G^\der$ for the derived subgroup of $G$ and $G^\ab=G/G^\der$ for its abelianization. We write $\det$ for the projection $G \rightarrow G^\ab$. 

Let $\M_{b,[\mu]}/\Spd \breve{E}([\mu])$ be the inscribed moduli of sections defined in \cite[\S9.4]{Howe.InscriptionTwistorsAndPAdicPeriods} (so that the diamond of the same name in the introduction is the underlying $v$-sheaf $\overline{\M_{b,[\mu]}}$). The construction of inscribed moduli of sections is functorial in the data $b$ and $[\mu]$ (it is given by a space of modifications between two $G$-bundles on the relative thickened Fargues-Fontaine curve, and the functoriality is via push-out), thus $\det$ induces a map 
\[ \mr{det}_{b,[\mu]}: \M_{b,[\mu]} \rightarrow \M_{\det(b), [\det(\mu)], \breve{E}([\mu])}.\] 

We fix a point $\tau: \Spd \mathbb{C}_p \rightarrow \M_{\det(b), \det(\mu), \breve{E}([\mu])}$ that lies in the image of $\det_{b,[\mu]}$ (such a point $\tau$ exists as a consequence of, e.g., \cite[Proposition 7.3.3]{HoweKlevdal.AdmissiblePairsAndpAdicHodgeStructuresIITheBiAnalyticAxLindemannTheorem}). We obtain an inscribed $v$-sheaf over $\Spd \mbb{C}_p$, $\M_{b,[\mu]}^\tau := \tau^* \M_{b,[\mu]}$,  
whose underlying $v$-sheaf $\overline{\M_{b,[\mu]}^\tau}$ is the diamond considered in the introduction. It admits period maps to the inscribed Schubert cells 
\[ \pi_{\Hdg}^{+,\tau}: \M_{b,[\mu]}^\tau \rightarrow \Gr_{[\mu], \mbb{C}_p} \textrm{ and } \pi_{\HT}^{+,\tau}: \M_{b,[\mu]}^\tau \rightarrow \Gr_{[\mu^{-1}], \mbb{C}_p} \]
induced from the maps $\pi_{\Hdg}^+$ and $\pi_{\HT}^+$ of \cite[\S9.4]{Howe.InscriptionTwistorsAndPAdicPeriods} on $\M_{b,[\mu]}$; the induced maps on underlying $v$-sheaves are the period maps that appeared in the introduction.  

\subsection{Computation of the tangent bundle}
It follows from the general formalism of \cite{Howe.InscriptionTwistorsAndPAdicPeriods} that $T_{\M_{b,[\mu]}^\tau}$ is
the pullback to $\M_{b,[\mu]}^\tau$ of the kernel of $d\det_{b,[\mu]}$, the derivative of $\det_{b,[\mu]}$.

We now describe this kernel. To that end, note that $\Lie G=\Lie G^{\der} \oplus \Lie Z(G)$, and the kernel of the derivative of $\det: G \rightarrow G^\ab$ at the identity element is $\Lie G^\der$. Writing $\mf{g}^\circ = \Lie G^{\der}(E)$ and $\mf{z} = \Lie Z(G)(E)$, we have $\mf{g}=\mf{g}^\circ \oplus \mf{z}$ as a representation of $G$ (via the adjoint action). 

We write $\mf{g}^\circ_b$ for $\mf{g}^\circ \otimes_E \breve{E}$ equipped with the isocrystal structure coming from the adjoint action of $b$. Over $\M_{b,[\mu]}$, we have a universal meromorphic isomorphism $\mc{E}_1 \rightarrow \mc{E}_b$ over $\mc{X}\backslash \infty$, where these are the trivial $G$-bundle and the $G$-bundle associated to $b$, respectively, on the relative thickened Fargues-Fontaine curve $\mc{X}$. If we push-out these bundles to $\mf{g}^\circ$, and restrict to a punctured formal neighborhood of the canonical Cartier divisor $\infty$ on $\mc{X}$, we obtain an isomorphism over $\M_{b,[\mu]}$,
\[ \mf{g}^\circ \otimes_E \mbb{B}_\dR = \mf{g}_b^\circ \otimes_{\breve{E}} \mbb{B}_\dR. \]
We write 
\[ \mf{g}^{\circ,+}_{\mr{max}}:=\mf{g}^\circ_b \otimes_{\breve{E}} \mbb{B}^+_\dR + \mf{g}^\circ \otimes_{E} \mbb{B}^+_\dR, \]
a locally free $\mbb{B}^+_\dR$-module, and form the modification $\mc{E}^\circ_{\max}$ of $\mf{g}^\circ \otimes_E \mathcal{O}_{\mc{X}}$ by this lattice. Recall that $\BC(\mathcal{E}^\circ_{\mr{max}})$ denotes the inscribed $v$-sheaf of global sections of this vector bundle. 

\begin{lemma}\label{lemma.tangent-tau} There is a canonical identification $T_{\M_{b,[\mu]}^\tau}=\ker d\det_{b,[\mu]}=\BC(\mathcal{E}^\circ_{\mr{max}})$.
\end{lemma}
\begin{proof}
We have already explained the first equality. To obtain the description of the kernel, note that \cite[Corollary 9.2.3]{Howe.InscriptionTwistorsAndPAdicPeriods} shows that $T_{\M_{b,[\mu]}}=\BC(\mc{E}_{\mr{max}})$, where $\mc{E}_{\mr{max}}$ is formed by the same construction starting with $\mf{g}$ instead of $\mf{g}^\circ$. The claim about the kernel then follows from the functoriality of the computations in \cite[Corollary 9.2.3]{Howe.InscriptionTwistorsAndPAdicPeriods} since $\mf{g}^\circ$ is the kernel of the derivative of $\det$ in $\mf{g}$. 
\end{proof}

\begin{lemma}\label{lemma.non-negative-slopes}
For any algebraically closed perfectoid $C/\mathbb{C}_p^\flat$ and $z: \Spd C \rightarrow \M_{b,[\mu]}^\tau$, the induced bundle $z^*\mc{E}^\circ_{\mr{max}}$ on $X_{E,C}^\alg$ has non-negative Harder-Narasimhan slopes.
\end{lemma}
\begin{proof}
Since $\mf{g}^\circ \otimes_E \mc{O}_{X_{E,C}^\alg}  \subseteq z^*\mc{E}^\circ_{\mr{max}}$ with torsion quotient, its Harder-Narasimhan slopes are non-negative --- indeed, otherwise $z^*\mc{E}^\circ_{\mr{max}}$ admits a nonzero morphism to a simple bundle $\mc{O}(\lambda)$ for $\lambda<0$, which would restrict to a non-zero morphism from the trivial bundle ${\mf{g}^\circ \otimes_E \mc{O}_{X_{E,C}^\alg}}$, but such a morphism does does not exist. 
\end{proof}

It will be the dual bundle that is naturally related to the geometry of fibers of $\pi_1 \times \pi_2$. We write $\mf{g}^{\circ,+}_{\mr{min}}:=\mf{g}^\circ_b \otimes_{\breve{E}} \mbb{B}^+_\dR \cap \mf{g}^\circ \otimes_{E} \mbb{B}^+_\dR$, and let $\mc{E}^{\circ}_{\mr{min}}$ be the associated modification of $\mf{g}^\circ \otimes_E \mc{O}_{\mc{X}}$.

\begin{lemma}\label{lemma.dual-of-e-max}
$(\mc{E}^\circ_{\mr{max}})^* \cong \mc{E}^\circ_{\mr{min}}$.
\end{lemma}
\begin{proof}
We have 
\[ (\mf{g}^{\circ,+}_{\max})^*=(\mf{g}^\circ_b \otimes_{\breve{E}} \mbb{B}^+_\dR + \mf{g}^\circ \otimes_{E} \mbb{B}^+_\dR)^* = (\mf{g}^\circ_b)^*\otimes_{\breve{E}} \mbb{B}^+_\dR \cap (\mf{g}^\circ)^* \otimes_{E} \mbb{B}^+_\dR. \]
It follows that $(\mc{E}^\circ_{\mr{max}})^*$ is the modification of $(\mf{g}^\circ)^* \otimes_E \mc{O}_{\mc{X}}$ associated to the lattice \[ (\mf{g}^\circ_b)^*\otimes_{\breve{E}} \mbb{B}^+_\dR \cap (\mf{g}^\circ)^* \otimes_{E} \mbb{B}^+_\dR.\]
Since $\mf{g}^\circ$ is a semisimple lie algebra, the Killing form gives an isomorphism $\mf{g}^\circ \cong (\mf{g}^\circ)^*$ as a representation of $G$. Thus we obtain an isomorphism of this modification with the modification of $\mf{g}^\circ \otimes_E \mc{O}_{\mc{X}}$ associated to the lattice $\mf{g}^\circ_b\otimes_{\breve{E}} \mbb{B}^+_\dR \cap \mf{g}^\circ \otimes_{E} \mbb{B}^+_\dR = \mf{g}^{\circ,+}_{\mr{\mathrm{min}}}$, i.e. $\mc{E}^\circ_{\mr{min}}.$ 
\end{proof}

\begin{remark}
\cref{lemma.dual-of-e-max}  will not typically hold if $G$ is not reductive; this is the only reason that we work with $G$ reductive instead of a general connected linear algebraic group. 
\end{remark}

\subsection{Fibers of $\pi_\Hdg^{+,\tau} \times \pi_\HT^{+,\tau}$}
We now relate our compuation to the structures of the fibers of $\pi^\tau:=\pi_{\Hdg}^{+,\tau} \times \pi_{\HT}^{+,\tau}$. 

\begin{lemma}\label{lemma.tau-torsor}
$\pi_\Hdg^{+,\tau}$ is a $G^\der(E^\lfid)$-torsor. 
\end{lemma}
\begin{proof}
Since $\pi_{\Hdg}^+$ is a $G(E^\lfid)$-torsor and the associated period map for $\mc{M}_{\det(b),\det([\mu])}$ is a $G^\mr{ab}(E^\lfid)$-torsor (over a point, since the group is abelian so the Schubert cells are all points), we find the fibers of $\pi_{\Hdg}^{+,\tau}$ are torsors for the kernel $G^\der(E^\lfid)$ of $G(E^\lfid) \rightarrow G^\mr{ab}(E^\lfid)$.  
\end{proof}

For $z: \Spd(C) \rightarrow \M_{b,[\mu]}^\tau$ a rank one geometric point, we write
\[ F_z:= \pi^\tau \times_{\Gr_{[\mu]} \times \Gr_{[\mu^{-1}]}} (\pi^\tau \circ z). \]
In other words, $F_z$ is the intersection of the fibers of $\pi_\Hdg^{+,\tau}$ and $\pi_{\HT}^{+,\tau}$ that contain $z$, equipped with its natural inscribed structure. 

Because $\mc{E}^\circ_{\mr{min}} \subseteq \mc{E}^\circ_{\mr{max}}$, there is a natural inclusion $\BC(\mc{E}^\circ_{\mr{min}}) \hookrightarrow \BC(\mc{E}^\circ_{\mr{max}}).$ 

\begin{lemma}\label{lemma.min-tangent-fiber}
The isomorphism $T_{\M_{b,[\mu]}^\tau}=\BC(\mathcal{E}^\circ_{\mr{max}})$ restricts to an isomorphism $T_{F_{z}}=\BC(\mathcal{E}^\circ_{\mr{min}}).$
\end{lemma}
\begin{proof}
    The tangent bundle $T_{F_z}$ is the restriction to $F_z$ of the kernel of $d \pi^{\tau} = d\pi_{\Hdg}^{+,\tau} \times d\pi_{\HT}^{+,\tau}$. From the computation of the derivatives of the period maps in \cite[Corollary 9.2.3]{Howe.InscriptionTwistorsAndPAdicPeriods} (see also the corresponding diagram in \cite[\S9.4]{Howe.InscriptionTwistorsAndPAdicPeriods}), this is exactly $\BC(\mathcal{E}^\circ_{\mr{min}})$: indeed, we have
        \[ \BC(\mc{E}^\circ_{\max}) \subseteq \mf{g} \otimes_E \mbb{B}_\dR = \mf{g}_b \otimes_{\breve{E}} \mbb{B}_\dR \]
    and the kernel of $d\pi_{\Hdg}^{+,\tau}$ will be exactly the intersection of 
       with $\mf{g} \otimes_E \mbb{B}^+_\dR$ while the kernel of $d\pi_{\HT}^{+,\tau}$ will be exactly the intersection with $\mf{g}^\circ_b\otimes_{\breve{E}} \mbb{B}^+_\dR$, so that together we obtain the intersection with $\mf{g}^{\circ,+}_{\mr{min}}$, which is $\BC(\mc{E}^\circ_{\min})$.  
\end{proof}

We note that $F_z \hookrightarrow \pi_{\Hdg}^{+,\tau} \times_{\Gr_{[\mu]}}\pi_{\Hdg}^{+,\tau}\circ z$. Since $\pi_\Hdg^{+,\tau}$ is a $G^\der(E^\lfid)$-torsor by \cref{lemma.tau-torsor}, we obtain an embedding over $\Spd C$, $F_z \hookrightarrow G^\der(E^\lfid) \cdot z$. In fact we can be more precise:

\begin{lemma}\label{lemma.fiber-orbit} For $H_z:= \mathrm{Stab}_{G^\der(E^\lfid)}(\pi_\HT^{+,\tau}(z))$,  $F_z = H_z \cdot z. $
\end{lemma}
\begin{proof} The induced map $(\pi^{+,\tau}_\Hdg)^{-1}(z)=G^\der(E^\lfid) \cdot z \xrightarrow{\pi_{\HT}^{+,\tau}} \Gr_{[\mu^{-1}]}$ sends $g \cdot z$ to $g \cdot \pi_\HT^{+,\tau}(z)$ (by the equivariance of $\pi_{\HT}^{+,\tau}$), so the claim is immediate from the definitions.
\end{proof}

\begin{lemma}\label{lemma.discrete-lie-translation}
$H_z(\Spd C) \leq G(E)$ is discrete if and only if $\Lie H_z(\Spd C) = 0$.
\end{lemma}
\begin{proof}
By construction, 
\[ H:=H_z(\Spd C) = \mr{Stab}_{G^\der(E)} (\pi_\HT^{+,\tau}(z)) \]
Since $H$ is a closed subgroup of $G^\der(E)$, it is a $p$-adic Lie group. Thus it is discrete if and only if $\Lie H=0$. By exponentiation of tangent vectors in a $p$-adic Lie group, we find this is equivalent to $\Lie H_z(\Spd C)=0$. 
\end{proof}

\begin{proposition}\label{prop.HN-slopes-condition}
For $z: \Spd(C) \rightarrow \M_{b,[\mu]}^\tau$, we write $z^*\mc{E}^\circ_{\mr{max}}$ for the induced bundle on $X_{E, C^\flat}^\alg$. The Harder-Narasimhan slopes of $z^*\mc{E}^\circ_{\mr{max}}$ are all non-negative, and the following are equivalent: 
\begin{enumerate}
\item $z^*\mc{E}^\circ_{\mr{max}}$ does not admit zero as a Harder-Narasimhan slope.
\item $z^*\mc{E}^\circ_{\mr{min}}$ does not admit zero as a Harder-Narasimhan slope. 
\item $T_{F_z, z}(\Spd C) = 0$.
\item $(\Lie H_z)(\Spd C) = 0$
\item $H_z(\Spd C)=\mr{Stab}_{G^\der(E)} (\pi_{\HT}^{+,\tau}(z))$ is a discrete subgroup of $G^\der(E)$. 
\item $(F_z)(\Spd C)$ is discrete subset of $G^\der(E) \cdot z$. 
\end{enumerate}
\end{proposition}
\begin{proof}
The non-negativity of slopes is \cref{lemma.non-negative-slopes}. 

The bundle $z^*(\mc{E}^\circ_{\mr{max}})$ admits zero as a Harder-Narasimhan slope if and only if its dual does, and this dual is equal to $z^*\mc{E}^\circ_{\mr{min}}$ by \cref{lemma.dual-of-e-max}, so we obtain the equivalence between (1) and (2).

Since $z^*\mc{E}^\circ_{\mr{min}}$ is the dual of a bundle with non-negative slopes, its slopes are non-positive, thus it admits zero as a slope if and only if it has a non-zero global section. Invoking \cref{lemma.min-tangent-fiber}, this gives the equivalence between (2) and (3).

By \cref{lemma.fiber-orbit}, we obtain $T_{F_z,z}=\Lie H_z$, which gives the equivalence between (3) and (4). The equivalence between (4) and (5) follows from \cref{lemma.discrete-lie-translation}. Finally, the equivalence between (5) and (6) is again a consequence of \cref{lemma.fiber-orbit}.
\end{proof}

Following the definition in the introduction, a geometric point $z$ is called very special if it does not satisfy any of the equivalent conditions of \cref{prop.HN-slopes-condition}.

\begin{remark}\label{remark.connected}
    The Banach-Colmez space of global sections of $z^*\mc{E}^\circ_{\mr{max}}$ appearing in \cref{prop.HN-slopes-condition}, i.e. the Tangent Space, is connected if and only if zero does not appear as a slope. Thus this can also be viewed as a criterion for connectedness of the Tangent Space. 
\end{remark}

\begin{corollary}\label{cor.vsp-closed}
    There is a unique closed subdiamond $\overline{\M_{b,[\mu]}^{\tau}}^{\vsp} \subseteq \overline{\M_{b,[\mu]}^{\tau}}$ whose rank one geometric points are exactly the very special points. 
\end{corollary}
\begin{proof}
    We first note that, by \cref{prop.HN-slopes-condition}, a point $z: \Spd C \rightarrow \M_{b,[\mu]}^\tau$ is very special if and only if $\mc{E}^\circ_{\max,z}$ admits $0$ as a Harder-Narasimhan slope. Since the slopes are non-negative, this is a closed condition by the semi-continuity of the Harder-Narasimhan polygon  (\cite[Theorem I.3.4]{FarguesScholze.GeometrizationOfTheLocalLanglandsCorrespondence}).
\end{proof}

We write $\overline{\M_{b,[\mu]}^{\tau}}^{\nvsp}$ for the open complement of the very special locus. 

\section{The EL case}\label{s.EL-case}

In this section we will prove \cref{theorem.minuscule-EL} (see \cref{corollary.non-sp-locus}). 

\subsection{Setup}
In this section $E=\mathbb{Q}_p$. Following \cite[\S2.2]{IvanovWeinstein.TheSmoothLocusInInfiniteLevelRapoportZinkSpaces} (with a slight modification --- see \cref{remark.hodge-cocharacter-change}), we fix EL data $\mathcal{D}=(B,V,H,[\mu])$: $B$ is a semisimple $\mathbb{Q}_p$-algebra, $V$ is a finite dimensional $B$-module, $H$ is a $p$-divisible group up to isogeny over $\fpbar$ equipped with an action $B \hookrightarrow \End(H)$, and $[\mu]$ is a conjugacy class of cocharacters of $G_{\overline{\mathbb{Q}}_p}$ for $G:=\GL_B(V)$ such that 
\begin{enumerate}
\item The rational covariant Dieudonn\'{e} module $M(H)$, as a $B \otimes_{\qp} \qpbreve$-module, is isomorphic to $V \otimes \qpbreve$
\item The action on $V \otimes \qpbar$ of $\mathbb{G}_m$ induced by any $\mu \in [\mu]$ is by weights $0$ and $1$, and the weight $1$ space has dimension equal to that of $H$.
\end{enumerate}
\begin{remark}\label{remark.hodge-cocharacter-change}
In \cite{IvanovWeinstein.TheSmoothLocusInInfiniteLevelRapoportZinkSpaces} they also take the cocharacter to have weights $0$ and $1$, but require the weight $0$ subspace have dimension equal to that of $H$. We have chosen our $\mu$ as the homological Hodge cocharacter, i.e. to lie in the conjugacy class that would act as multiplication by $z$ on the Lie algebra of the associated lift of $H$ (note that for the Hodge filtration on $M(H)$ associated to a lift, $\gr^{-1}M(H)$ is the Lie algebra; the Hodge filtration is thus the filtration attached to $\mu^{-1}$). 
\end{remark}

We fix an isomorphism as in (1), so that the Frobenius on $M(H)=V \otimes \qpbreve$ is of the form $b \sigma$ for $b \in \GL_B(V)$. Thus we can consider the inscribed moduli of modifications $\M_{b,[\mu]}$ over $\qpbreve([\mu])$ as in \cref{s.analysis}. Recall that in this minuscule case, $\pi_\HT=\pi_{\HT}^+$ and $\pi_{\Hdg}=\pi_\Hdg^+$ as the Bialynicki-Birula map is an isomorphism.

\subsection{Construction}We note that there is a $\mu \in [\mu]$ defined over $\qpbreve([\mu])$ --- we fix such a $\mu$, and write $V[-1] \subseteq V \otimes \qpbreve([\mu])$ for the associated $-1$ weight space. 

\begin{lemma}\label{lemma.functor-vb-mod}
$\mathcal{M}_{b,[\mu]}$ is isomorphic over $\Spd \qpbreve([\mu])$ to the functor sending a test object $(\mc{X}/X_P, P/\Spd \qpbreve([\mu]))$ to the set of exact sequences of $B \otimes \mc{O}_{\mc{X}}$-modules 
\[ 0 \rightarrow V \otimes \mc{O}_{\mc{X}} \rightarrow \mc{E}_b(V) \rightarrow \infty_* W \rightarrow 0\]
where $W$ is a $\mc{O}_{\mc{X}_{P^\sharp}} \otimes_{\qp} B$-module locally isomorphic to ${V[-1] \otimes_{\qpbreve([\mu])} \mc{O}_{\mc{X}_{P^\sharp}}}$. 
\end{lemma}
\begin{proof}
Given a modification $\mc{E}_1 \rightarrow \mc{E}_b$ of type $\mu$, we obtain by push-out along the representation $V$ such an exact sequence. The inverse is given by sending such an exact sequence to the induced modification of $G$-torsors
\[ \mc{E}_1 = \mc{I}som_{B \otimes \mc{O}_{\mc{X}}}(V \otimes \mc{O}_{\mc{X}}, V \otimes \mc{O}_{\mc{X}}) \dashrightarrow \mc{I}som_{B \otimes \mc{O}_{\mc{X}}}(V \otimes \mc{O}_{\mc{X}}, \mc{E}_b(V)) = \mc{E}_b.  \]
\end{proof}

Note that the underlying $v$-sheaf of the functor in \cref{lemma.functor-vb-mod} is precisely that functor appearing in \cite[Proposition 3.2.3]{IvanovWeinstein.TheSmoothLocusInInfiniteLevelRapoportZinkSpaces}. 

\subsubsection{}We fix a determinant $\tau$. Then, in \cite[\S 5.6]{IvanovWeinstein.TheSmoothLocusInInfiniteLevelRapoportZinkSpaces}, Ivanov and Weinstein construct a smooth quasi-projective scheme $Z/X_{\mathbb{C}_p^\flat}^\alg$ whose associated moduli of sections is isomorphic to the subfunctor corresponding to $\overline{\M_{b,[\mu]}^\tau}$. In fact, with no further work, we obtain (using the moduli of sections construction of \cite[\S4.4]{Howe.InscriptionTwistorsAndPAdicPeriods}):
\begin{proposition}\label{prop.FS-moduli}
    For $Z/X_{\mathbb{C}_p^\flat}$ the smooth quasi-projective scheme constructed in \cite[\S 5.6]{IvanovWeinstein.TheSmoothLocusInInfiniteLevelRapoportZinkSpaces}, $(Z/X_{\mathbb{C}_p^\flat}^\alg)^\lfid =(Z^\an/X_{\mathbb{C}_p^\flat})^\lfid = \M_{b,[\mu]}^\tau.$ 
\end{proposition}
\begin{proof}
That the smooth quasi-projective scheme and its analytification have the same moduli of sections is always true. The identification with $\M_{b,[\mu]}^\tau$ is immediate from the construction in \cite[\S 5.6]{IvanovWeinstein.TheSmoothLocusInInfiniteLevelRapoportZinkSpaces}, which is by analytification of a smooth scheme that parameterizes such exact sequences with determinant $\tau$ for any scheme over $X_{\mathbb{C}_p^\flat}^\alg$ such that the pullback of the canonical divisor $\infty$ is a Cartier divisor. 
\end{proof}

\begin{corollary}\label{corollary.non-sp-locus}
The map $\overline{\M_{b,[\mu]}^\tau}^{\nvsp} \rightarrow \Spd \mathbb{C}_p$ is cohomologically smooth. 
\end{corollary}
\begin{proof}
We claim that $\overline{\M_{b,[\mu]}^\tau}^{\nvsp}$ is the locus $\mc{M}_Z^\sm$ appearing in \cite[Definition IV.4.1]{FarguesScholze.GeometrizationOfTheLocalLanglandsCorrespondence}. Given this claim, the cohomological smoothness follows from the Jacobian criterion \cite[Theorem IV.4.2]{FarguesScholze.GeometrizationOfTheLocalLanglandsCorrespondence}. To see this claim it suffices to check at rank one points $z$ as above. For these note that, for $s$ the associated section of $Z^\an$, \cite[Example 4.4.3-(2)]{Howe.InscriptionTwistorsAndPAdicPeriods} implies $\BC(s^*T_{Z^\an/X_{\mathbb{C}_p^\flat}})=T_{(Z/E)^\lfid,z}.$ But, by \cref{prop.FS-moduli} and \cref{lemma.tangent-tau}, this is also equal to $\BC(z^*\mc{E}^\circ_{\mr{max}})$. Since $\BC(s^*T_{Z^\an/X_{\mathbb{C}_p^\flat}}) = \BC((s^*T_{Z^\an/X_{\mathbb{C}_p^\flat}})_{\geq 0})$, where the subscript denotes the non-negative slope part, by full faithfulness of the functor from vector bundles of non-negative slope to Banach-Colmez spaces (a part of \cite[Th\'{e}or\`{e}me 1.2]{LeBras.EspacesDeBanachColmezEtFaisceauxCoherentsSurLaCourbeDeFarguesFontaine}), we deduce $(s^*T_{Z^\an/X_{\mathbb{C}_p}^\flat})_{\geq 0}=z^*\mc{E}^\circ_{\mr{max}}$. But it is straightforward to check that both $z^*\mc{E}^\circ_{\mr{max}}$ and $s^*T_{Z^\an/X_{\mathbb{C}_p^\flat}}$ have dimension equal to the dimension of $G^\der$, thus we conclude $z^*\mc{E}^\circ_{\mr{max}}=s^*T_{Z^\an/X_{\mathbb{C}_p^\flat}}$. The above discussion and \cref{prop.HN-slopes-condition} shows that the space $\mc{M}_Z^{\sm}$ where $s^*T_{Z^\an/X_{\mathbb{C}_p^\flat}}=z^*\mc{E}^\circ_{\mr{max}}$ has strictly positive slopes is equal to the space $\overline{\M_{b,[\mu]}^\tau}^{\nvsp}$, as desired. 
\end{proof}

\subsection{Endomorphism characterization of very special points}
Consider a rank one geometric point $z: \Spd C \rightarrow \M_{b,[\mu]}$. Then, via the Scholze-Weinstein classification \cite{ScholzeWeinstein.ModuliOfpDivisibleGroups}, $\pi_\HT(z)$ defines a $p$-divisible group up-to-isogeny $G_z$ over $\mathcal{O}_{C^\sharp}$ with an inclusion $B \hookrightarrow \End(G_z)=: A_x$. We write $A_z=\End_B(G)$. For $F$ the center of $B$, we have $F \hookrightarrow A_z$. 

\begin{lemma}\label{lemma.vsp-pdiv}With notation as above, $z$ is very special if and only if $F\hookrightarrow A_z$ is not an isomorphism.
\end{lemma}
\begin{proof}
    By \cref{prop.HN-slopes-condition}, being very special is equivalent to $(\Lie H_z)(\Spd C) \neq 0$.  But, essentially by definition, nonzero elements of $\Lie H_z(\Spd C)$ correspond to noncentral elements in the $B$-linear endomorphisms of $V$ preserving the Hodge-Tate filtration (using that $\Lie \GL_B(V)=\End_B(V) = F \oplus (\Lie \GL_b(V))^\circ$); by the Scholze-Weinstein \cite{ScholzeWeinstein.ModuliOfpDivisibleGroups} classification, these give elements of $A_z \backslash F$.  
\end{proof}

The condition of \cref{lemma.vsp-pdiv} is taken as the definition of special in \cite[Definition 3.5.2]{IvanovWeinstein.TheSmoothLocusInInfiniteLevelRapoportZinkSpaces}. 

\bibliographystyle{plain}
\bibliography{references, preprints}

\end{document}